\documentclass[11pt,thmsa]{article}%
\usepackage{amsmath}
\usepackage{amsfonts}
\usepackage{amssymb}
\usepackage{graphicx}%
\newtheorem{theorem}{Theorem}[section]

\newtheorem{corollary}[theorem]{Corollary}

\newtheorem{lemma}[theorem]{Lemma}

\newenvironment{proof}[1][Proof]{\noindent\textbf{#1.} }{\ \rule{0.5em}{0.5em}}

\begin{document}
\title{Killing frames and S-curvature of  homogeneous Finsler spaces\thanks{Supported by NSFC (no. 11221091, 11271198, 11271216) and SRFDP of China}}
\author{Ming Xu$^1$ and Shaoqiang Deng$^2$ \thanks{Corresponding author. E-mail: dengsq@nankai.edu.cn}\\
\\
$^1$College of Mathematics\\
Tianjin Normal University\\
 Tianjin 300387, P.R. China\\
 \\
$^2$School of Mathematical Sciences and LPMC\\
Nankai University\\
Tianjin 300071, P.R. China}
\date{}
\maketitle
\begin{abstract}
In this paper, we first deduce a formula of S-curvature of homogeneous Finsler spaces in terms of Killing vector fields. Then we prove that a homogeneous
Finsler space has isotropic S-curvature if and only if it has vanishing S-curvature. In the special case that the homogeneous Finsler space is a Randers space, we give an explicit formula which coincides with the previous formula obtained by the second author using other methods.



\end{abstract}
\section{Introduction}
The notion of S-curvature was introduced by Z. Shen in \cite{Sh1} in his study of volume comparison in Finsler geometry. S-curvature is an important quantity in Finsler geometry in that it has some mysterious interrelations with other quantities such as flag curvature, Ricci scalar, etc.
Shen showed that the Bishop-Gromov volume  comparison theorem holds for Finsler spaces with vanishing S-curvature. Therefore, it is also significant to characterize Finsler spaces with vanishing S-curvature.

The goal of this article is to give an explicit formula of  S-curvature of a homogeneous Finsler space in terms of Killing vector fields.
Let $(M, g)$ be a connected Finsler space. Then the group of isometries of $(M, F)$, denoted by $I(M, F)$, is a Lie transformation group on $M$
with respect to the compact-open topology (see \cite{DH}).  A vector field $X$ on a Finsler space $(M, F)$ is called a Killing vector field, if the local one-parameters groups of transformations generated by $X$ consists of local isometries of $(M, F)$. A Killing vector field $X$ of $(M,F)$ can be equivalently described as follows. Any vector field $X$ on $M$ can naturally define a vector field $\tilde{X}$ on $TM$. The vector field $X$ generates a flow of diffeomorphisms $\rho_t$ on $M$, with
the corresponding flow of diffeomorphisms $\tilde{\rho_t}$ on $TM$. Then the value of $\tilde{X}$ at $(x,y)\in TM$ is just $\frac{d}{dt}[\tilde{\rho_t}(x,y)]|_{t=0}$. Obviously,  $X$
is a Killing vector field for $F$ if and only if $\tilde{X}(F)=0$.

 The space $(M, F)$ is called homogeneous if the action of $I(M, F)$ on $M$ is transitive.
In this case, $M$ can be written as a coset space $I(M, F)/I(M, F)_x$, where $I(M, F)_x$ is the isotropic subgroup of $I(M, F)$ at a  point of $x\in M$. Since $M$ is connected,
the unit connected component of $I(M, F)$, denoted by $G$, is also transitive on $M$. Let $H$ be the isotropic of $G$ at the point $x$. Then we have
$M=G/H$. Moreover, the Finsler metric $F$ can be viewed as an $G$-invariant Finsler metric on $G/H$.

Since $(M, F)$ is homogeneous, given any tangent vector $v\in T_y(M)$, $y\in M$, there exists a Killing vector field $X$ such that $X|_y=v$. Therefore, if we can get a formula for $S(X)$, where $X$ is an arbitrary Killing vector field, then the S-curvature of $(M, F)$ is completely determined.
The main result of this paper is a formula of the S-curvature as described above. As an application, we show that a homogeneous Finsler space has
istropic S-curvature if and only if it has vanishing S-curvature. This generalizes the similar results on homogeneous Randers spaces and Homogeneous $(\alpha,\beta)$-spaces in \cite{De} and \cite{DW}.

In Section 2, we present some preliminaries on Finsler spaces and S-curvature. Section 3 is devoted to  deducing the formula of  S-curvature. In Section 4, we apply our formula to homogeneous Randers spaces, and show that the formula coincides with the previous one obtained
in \cite{De} in this special case.
\section{Preliminaries}
In this section we recall some known results on Finsler spaces, for details we refer the readers to \cite{BCS}, \cite{CS} and \cite{Sh2}.

A Finsler metric on a manifold $M$ is a function $F:TM\backslash\{0\} \rightarrow {\mathbb R}^{+}$ satisfying the following properties:
\begin{enumerate}
\item\quad $F$ is  smooth  on $TM\backslash \{0\}$.
\item\quad $F$ is positively homogeneous of degree 1, namely, $F(\lambda y)=\lambda F(y)$, for any $\lambda >0$ and $y\in TM\backslash\{0\}$.
\item\quad For any standard local coordinate system $(TU, (x,, y)$ of $TM$, where $x=(x^i)$ in an small open neighborhood $U\subset M$, and $y=y^i\partial_{x^i}\in TM_x$, the fundamental tensor $g_{ij}(y)=\frac{1}{2}[F^2]_{y^i y^j}$ is positive definite whenever $y\neq 0$.
\end{enumerate}

Riemannian metrics are a special class of Finsler metrics widely studied by mathematicians. Their fundamental tensors only depends on $x$, which are regarded as the metrics themselves.

Randers metrics are the most well-known non-Riemannian Finsler metric. They are defined as $F=\alpha+\beta$, where $\alpha$ is a Riemannian metric and $\beta$ is a 1-form,
whose $\alpha$-norm $||\beta (x)||_\alpha$ is less than $1$ everywhere. Randers metrics are  generalized to $(\alpha,\beta)$-metrics of the form $F=\alpha\phi(\beta/\alpha)$.

\section{Killing frames and the geodesic spray}
A Killing frame for a Finsler manifold $(M,F)$ is a set of local vector fields $X_i$,
$i=1,\ldots,n=\dim M$, defined on an open subset $U$  around a given point,  such that
\begin{enumerate}
\item\quad The values $ X_i(x)$, $\forall i$, give bases for each tangent space $T_x(M)$, $x\in U$, and
\item\quad In $U$, each $X_i$ satisfies $\tilde{X}_i(F)=0$, in other words, the $X_i's$ are local Killing vector fields in $U$.
\end{enumerate}
Though Killing frames are rare in the general study of Finsler geometry, they can be easily found for a homogeneous Finsler space at any given point. Let the homogeneous Finsler space $(M,F)$ be presented as $M=G/H$, where $H$ is the isotropy subgroup
for the given $x$. The tangent space $TM_x$ can be identified as the quotient
$\mathfrak{m}=\mathfrak{g}/\mathfrak{h}$,  where $\mathfrak{g}$ and $\mathfrak{h}$ are
the Lie algebras of $G$ and $H$,  respectively. Take any basis $\{ v_1,\ldots,v_n \}$ of $\mathfrak{m}$, with the pre-images $\{\hat{v}_1,\ldots,\hat{v}_n\}$ in $\mathfrak{g}$. Then the Killing vector fields $\{X_1,\ldots,X_n\}$ on $M$ corresponding to $\hat{v}_i$s defines  a Killing frame around $x$. The choice of $\hat{v}_i$s or $X_i$s identifies the quotient space $\mathfrak{m}$ with a subspace of $\mathfrak{g}$, and then we can write
the decomposition of linear space
\begin{equation}
\mathfrak{g}=\mathfrak{h}+\mathfrak{m}.
\end{equation}

For the Killing frame $\{X_1,\ldots,X_n\}$ around $x\in M$, a set of $y$-coordinates $y=(y^i)$
can be defined by $y={y}^i X_i$.  Accordingly, we have the fundamental tensor
${g}_{ij}=\frac{1}{2}{[F^2]}_{{y}^i {y}^j}$, and the inverse matrix
of $({g}_{ij})$ is denoted as $({g}^{ij})$. When both the Killing frame and
the local coordinates $\{x=(\bar{x}^{\bar{i}}),y=\bar{y}^{\bar{j}}\partial_{\bar{x}^{\bar{j}}}\}$ are used,  the terms and indices for the local coordinates are marked with bars, and the indices with bars are moved up and down by the fundamental tensors  $\bar{g}^{\bar{i}\bar{j}}$ or $\bar{g}_{\bar{i}\bar{j}}$ for the local coordinates.
Let $f^i_{\bar{i}}$ and $f^{\bar{i}}_i$, $\forall i$ and $\bar{i}$, be the transition functions such that around $x$,
\begin{eqnarray}
\partial_{\bar{x}^{\bar{i}}} = f^i_{\bar{i}} X_i \mbox{ and }
X_i = f^{\bar{i}}_i\partial_{\bar{x}^{\bar{i}}}.
\end{eqnarray}
We summarize some easy and useful identities which show how the transition functions exchange the indices with and without bars:
\begin{eqnarray}
\bar{y}^{\bar{i}}=f^{\bar{i}}_i y^i &\mbox{ and }& y^i=f^i_{\bar{i}}\bar{y}^{\bar{i}}\\
\partial_{\bar{y}^{\bar{i}}}=f^{i}_{\bar{i}} \partial_{y^i} &\mbox{ and }&
\partial_{y^i}=f^i_{\bar{i}}\partial_{\bar{y}^{\bar{i}}},\\
\bar{g}_{\bar{i}\bar{j}}=f^i_{\bar{i}}g_{ij}f^j_{\bar{j}}&\mbox{ and }&
g_{ij}=f^{\bar{i}}_i \bar{g}_{\bar{i}\bar{j}} f^{\bar{j}}_j,\\
\bar{g}^{\bar{i}\bar{j}}=f^{\bar{i}}_i g^{ij} f^{\bar{j}}_j&\mbox{ and }&
g^{ij}=f^{i}_{\bar{i}}\bar{g}^{\bar{i}\bar{j}}f^j_{\bar{j}}.
\end{eqnarray}

To apply Killing frames to the study of Finsler geometry, we start with the geodesic spray.

\begin{theorem}\label{main-theorem-1}
Let $\{X_1,\ldots,X_n\}$ be a Killing frame around $x\in M$ for the Finsler metric $F$.
Then for $y=\tilde{y}^i X_i(x)\in TM_x$, the geodesic spray $G(x,y)$ can be presented as
\begin{equation}\label{geodesic-spray-formula}
G(x,y)={y}^i \tilde{X}_i+\frac{1}{2}g^{il}c^k_{lj}[F^2]_{y^k}y^j\partial_{y^i},
\end{equation}
where $c^k_{lj}$ are defined by $[X_l,X_j](x)=c^k_{lj}X_k(x)$.
\end{theorem}

 If we use the local coordinates $\{x=(\bar{x}^i)$ and $y=\bar{y}^{\bar{i}}\partial_{\bar{x}^{\bar{i}}}\}$, then a direct calculation shows that $c^k_{ij}$s can be presented as
\begin{equation}\label{0}
c^k_{ij}=[(f^{\bar{i}}_i\partial_{\bar{x}^{\bar{i}}}f^{\bar{j}}_j-
f^{\bar{i}}\partial_{\bar{x}^{\bar{i}}} f^{\bar{j}}_i)f^{k}_{\bar{j}}](x).
\end{equation}
Now consider the case that $M=G/H$ is a homogeneous Finsler space, where $H$ is the isotropy group of $x\in M$. Let the Killing vector fields $X_i$'s be defined by $\hat{v}_i\in\mathfrak{g}$,
$\forall i$. Then the tangent space $TM_x$
can be identified with the $n$-dimensional subspace $\mathfrak{m}$ spanned by the values of  all the $\hat{v}_i$'s at $x$. With respect to the decomposition
$\mathfrak{g}=\mathfrak{h}+\mathfrak{m}$,  there is a projection map ${\rm pr}:\mathfrak{g}\rightarrow\mathfrak{m}$. Note that for  the bracket operation $[\cdot,\cdot]$ on $\mathfrak{g}$, we have
$[\cdot,\cdot]_{\mathfrak{m}}={\rm pr}[\cdot,\cdot]$. Then $c^k_{ij}$s can be determined by
\begin{equation}\label{-2}
[\hat{v}_i,\hat{v}_j]_{\mathfrak{m}}=-c^k_{ij}\hat{v}_k.
\end{equation}

The proof of Theorem \ref{main-theorem-1} needs local coordinates $\{x=(\bar{x}^i)$ and $y=\bar{y}^{\bar{i}}\partial_{\bar{x}^{\bar{i}}}\}$ around $x$.
We first need to see how to present each $\tilde{X}_i$ with the local coordinates.

\begin{lemma}
For any vector field $X=f^{\bar{i}}\partial_{\bar{x}^{\bar{i}}}$ around $x$,
\begin{equation}\label{-1}
\tilde{X}(x,y)=f^{\bar{i}}\partial_{\bar{x}^{\bar{i}}}+\bar{y}^{\bar{i}}
\partial_{\bar{x}^{\bar{i}}}f^{\bar{j}}\partial_{\bar{y}^{\bar{j}}},
\end{equation}
for any $y=\bar{y}^{\bar{i}}\partial_{\bar{x}^{\bar{i}}}\in TM_x$.
\end{lemma}
\begin{proof}
Let $\rho_t$ and $\tilde{\rho}_t$ be the flows of diffeomorphisms $X$ generates on $M$ and $TM$, respectively. For each $i$, the flow curve $\tilde{\rho}_t(\partial_{\bar{x}^{\bar{i}}}|_x)$ can be presented as
\begin{equation}
(\rho_t(x),\partial_{\bar{x}^{\bar{i}}}+t\partial_{\bar{x}^{\bar{i}}} f_{\bar{j}}\partial_{\bar{x}^{\bar{j}}}+o(t)),
\end{equation}
so the flow curve $\tilde{\rho}_t(x,y)$ for $y=\bar{y}^{\bar{i}}\in TM_x$
has the local coordinates
\begin{equation}
(\rho_t (x), [\bar{y}^{\bar{j}}+
t\bar{y}^{\bar{i}}\partial_{\bar{x}^{\bar{i}}}f^{\bar{j}}+o(t)]\partial_{\bar{x}^{\bar{j}}}).
\end{equation}
Differentiating  with respect to $t$ and considering the values  at $t=0$, we get (\ref{-1}).
\end{proof}

Now we use the above lemma to recalculate the terms of the geodesic spray
\begin{equation}
G=\bar{y}^{\bar{i}}\partial_{\bar{x}^{\bar{i}}}-\frac{1}{2}\bar{g}^{\bar{i}\bar{l}}
([F^2]_{\bar{x}^{\bar{j}}\bar{y}^{\bar{l}}}\bar{y}^{\bar{j}}-[F^2]_{\bar{x}^{\bar{l}}})
\partial_{\bar{y}^{\bar{i}}}.
\end{equation}
By (\ref{-1}) and the property that $\tilde{X}_i(F)=0$, $\forall i$, we have the following equations which hold on a neighborhood around $x$:
\begin{eqnarray}\label{1}
\bar{y}^{\bar{i}}\partial_{\bar{x}^{\bar{i}}}
&=& y^i f^{\bar{i}}_i\partial_{\bar{x}^{\bar{i}}}\nonumber\\
&=& y^i(\tilde{X}_i-\bar{y}^{\bar{i}}\partial_{\bar{x}^{\bar{i}}} f^{\bar{j}}_i \partial_{\bar{y}^{\bar{j}}})\nonumber\\
&=& y^i\tilde{X}_i-f^{\bar{i}}_k\partial_{\bar{x}^{\bar{i}}}f^{\bar{j}}_i f^l_{\bar{y}^{\bar{j}}}y^i y^k\partial_{y^l}\nonumber\\
&=& y^i\tilde{X}_i-f^{\bar{i}}_k\partial_{\bar{x}^{\bar{i}}}f^{\bar{j}}_j
f^i_{\bar{j}}y^j y^k\partial_{y^i},
\end{eqnarray}
\begin{eqnarray}\label{2}
\bar{g}^{\bar{i}\bar{l}}[ F^2 ]_{\bar{x}^{\bar{l}}}\partial_{\bar{y}^{\bar{i}}}
&=& g^{il}f^{\bar{l}}_l [F^2]_{\bar{x}^{\bar{l}}}\partial_{y^i}\nonumber\\
&=& -g^{il}(y^{\bar{i}}\partial_{\bar{x}^i} f^{\bar{j}}_l
[F^2]_{\bar{y}^{\bar{j}}})\partial_{y^i}\nonumber\\
&=& -g^{il} f^{\bar{i}}_j\partial_{\bar{x}^{\bar{i}}}f^{\bar{j}}_l f^k_{\bar{j}}
[F^2]_{y^k} y^j\partial_{y^i},
\end{eqnarray}
and
\begin{eqnarray}\label{3}
\bar{g}^{\bar{i}\bar{l}}[F^2]_{\bar{x}^{\bar{j}}\bar{y}^{\bar{l}}}\bar{y}^{\bar{j}}
\partial_{\bar{y}^{\bar{i}}} &=&
g^{il}[f^{\bar{j}}_j [F^2]_{\bar{x}^{\bar{j}}}]_{y^l} y^j\partial_{y^i}\nonumber\\
&=& -g^{il}(\bar{y}^{\bar{i}}\partial_{\bar{x}^{\bar{i}}} f^{\bar{j}}_j
[F^2]_{\bar{y}^j})_{y^l}y^j\partial_{y^i}\nonumber\\
&=& -g^{il} f^{\bar{i}}_l\partial_{\bar{x}^i}f^{\bar{j}}_j [F^2]_{\bar{y}^{\bar{j}}} y^j
\partial_{y^i}
-g^{il}f^{\bar{i}}_k\partial_{\bar{x}^i} f^{\bar{j}}_j f^h_{\bar{j}}
[F^2]_{y^h y^l} y^j y^k\partial_{y^i}\nonumber\\
&=& -g^{il}f^{\bar{i}}_l\partial_{\bar{x}^{\bar{i}}}f^{\bar{j}}_j f^k_{\bar{j}}
[F^2]_{y^k}y^j\partial_{y^i} - 2f^{\bar{i}}_k\partial_{\bar{x}^{\bar{i}}}
f^{\bar{j}}_j f^i_{\bar{j}} y^j y^k\partial_{y^i}.
\end{eqnarray}
By (\ref{1})-(\ref{3}) and  (\ref{0}), we get
\begin{eqnarray}
G(x,y) &=& y^i \tilde{X}_i+\frac{1}{2}g^{il}(f^{\bar{i}}_l
\partial_{\bar{x}}^{\bar{i}}-f^{\bar{i}}_j\partial_{\bar{x}^{\bar{i}}}f^{\bar{j}}_l
)f^k_{\bar{j}}[F^2]_{y^k}y^j\partial_{y^i}\nonumber\\
&=& y^i\tilde{X}_i+\frac{1}{2}g^{il}c^{k}_{lj}[F^2]_{y^k}y^j\partial_{y^i}.
\end{eqnarray}
This completes the proof of Theorem \ref{main-theorem-1}.

\section{The formula of S-curvature for a  homogeneous Finsler space}

The formula (\ref{geodesic-spray-formula}) of the geodesic spray can be immediately applied to get a formula of S-curvature. Suppose in a  local coordinate system,  $x=(\bar{x}^{\bar{i}})$ and
$y=\bar{y}^{\bar{i}}\partial_{\bar{x}^{\bar{i}}}\in TM_x$, with $y\neq 0$. Then the distortion
function is defined by
\begin{equation}
\tau(x,y)=\ln\frac{\sqrt{\det(\bar{g}_{\bar{p}\bar{q}})}}{\sigma(x)},
\end{equation}
where $\sigma(x)$ is defined by
$$\sigma(x)=\frac{\mbox{Vol}(B^n)}{\mbox{Vol}\{(y^i)\in {\Bbb
R}^n|F_x(y^ib_i)<1\}},$$
where Vol means the volume of a subset in
the standard Euclidean space ${\Bbb R}^n$ and $B^n$ is the open
ball of radius $1$. The function $\sigma(x)$ can be used to define the Busemann-Hausdorff volume $\sigma(x)d\bar{x}^1\cdots
\bar{x}^n$.
The S-curvature of the nonzero tangent vector $(x,y)$, denoted as $S(x,y)$,  is defined to be  the derivative of $\tau$ in
the direction of the geodesics of $G(x,y)$, with initial vector $y$.

Notice that the distortion function $\tau(x,y)$ is only determined by the metric $F$, not
relevant to the choice of local coordinates or frames.
If there is Killing frame $\{X_1,\ldots,X_n\}$ around $x$, then it is not hard to see that
$\tilde{X}_i \tau=0$, $\forall i$. Thus only the derivative of $\tau$ in the direction
$\frac{1}{2}g^{il}c^k_{lj}[F^2]_{y^k}y^j\partial_{y^i}$ remains to appear in the S-curvature formula.
Notice also that in the expression of $\tau$, $\sigma(x)$ is a function of $x$  only. This observation leads to   the following formula for the S-curvature.
\begin{theorem}
Let $\{X_1,\ldots,X_n\}$ be a Killing frame around $x$, then for any $y\neq 0$ in $TM_x$, the S-curvature at $(x,y)$ can be presented with the notations for the Killing frame as
\begin{equation}\label{S-curvature-formula-1}
S(x,y)=\frac{1}{2}g^{il}c^k_{lj}[F^2]_{y^k}y^j I_i,
\end{equation}
where $I_i=[\ln\sqrt{\det(g_{pq})}]_{y^i}$ are the coefficients of the mean Cartan torsion with respect to the basis the Killing frame induced in $TM_x$.
\end{theorem}

Now assume $M=G/H$ is homogeneous, with $H$ being the isotropy group at $x$. In Section 3 we have seen the existence of Killing frames around $x$. Each Killing frame $\{X_1,\ldots,X_n\}$ determines a decomposition $\mathfrak{g}=\mathfrak{h}+\mathfrak{m}$,
where $X_i$ is determined by $\hat{v}_i$ in $\mathfrak{m}$. Let
the operation $[\cdot,\cdot]_\mathfrak{m}$ be defined as before.
The gradient field of $\ln\sqrt{\det(g_{pq})}$ with respect to the fundamental tensor on $TM_x\backslash 0$ is the $\mathfrak{m}$-valued function
\begin{equation}
g^{il}I_i \hat{v}_l=g^{il}[\ln\sqrt{\det(g_{pq})}]_{y^i}\hat{v}_l.
\end{equation}
We will denote it as $\nabla^{g_{ij}} \ln\sqrt{\det(g_{pq})}(y)$ for $y\in\mathfrak{m}$. Let $\langle\cdot , \cdot\rangle_y$ be the inner product defined by the fundamental tensor $g_{ij}$ at $y$. Then by (\ref{-2}) we can rewrite (\ref{S-curvature-formula-1}) as
\begin{eqnarray}\label{S-curvature-formula-2}
S(x,y) &=& g^{il}c^k_{lj}g_{kh} y^h y^j I_i\nonumber\\
&=& \langle [y,\nabla^{g_{ij}}\ln\sqrt{\det(g_{pq})}(y)]_\mathfrak{m},y\rangle_y,
\end{eqnarray}
which gives a more beautiful formula for the S-curvature of a homogenous Finsler space.
The formula (\ref{S-curvature-formula-2}) is relevant to the choice of $\mathfrak{m}$
rather than the specified basis of $\mathfrak{m}$ to generate the Killing frame. To summarize, we have proved the following
\begin{theorem}\label{main-theorem-2}
Let $M$ be a homogeneous Finsler space $G/H$, where $H$ is the isotropy
subgroup of $x\in M$. Fix any complement $\mathfrak{m}$ of $\mathfrak{h}$ in $\mathfrak{g}$, with the corresponding $[\cdot,\cdot]_\mathfrak{m}$. Then for any nonzero $y\in \mathfrak{m}=TM_x$, we have
\begin{equation}
S(x,y)=\langle [y,\nabla^{g_{ij}}\ln\sqrt{\det(g_{pq})}(y)]_\mathfrak{m},y\rangle_y.
\end{equation}
\end{theorem}

A Finsler metric $F$ is said to have isotropic S-curvature when the S-curvature $S=(n+1)c(x)F$ for some function $c(x)$ on $M$. An immediate application of Theorem \ref{main-theorem-2} is
the following corollary.
\begin{corollary}
A homogeneous Finsler space is of isotropic S-curvature if and only if it has vanishing  S-curvature.
\end{corollary}
\begin{proof}
We  need only consider the S-curvature at a fixed point $x$.
The function $\ln\sqrt{\det(g_{pq})}$ is homogeneous of degree 0, so it must reach its maximum or minimum at some nonzero $y$, where the gradient field vanishes. Then by (\ref{S-curvature-formula-2}), $S(x,y)=0$. If the S-curvature is isotropic, i.e.,if $S=(n+1)c(x)F$, then $c(x)$ must be 0. This proves the "only if" part of the corollary.  The "if" part is obvious.
\end{proof}

As another corollary, we obtain a important property of homogeneous Einstein Finsler spaces. Recall that an Einstein Finsler space must have constant S-curvature (see \cite{BR}). Therefore we have

\begin{corollary}
A Homogeneous Einstein Finsler space must have vanishing S-curvature.
\end{corollary}

\section{Homogeneous Randers space}

In this last section, we will apply the S-curvature formula (\ref{S-curvature-formula-2}) to calculate the S-curvatures of
homogeneous Randers spaces.  It turns out that the resulting formula coincides with the one given in \cite{De}. Although the situation will be a little more complicated, the same technique can be  transported to the $(\alpha,\beta)$-metrics, and we can get the same formula as in \cite{DW}. Though the calculation here seems longer, it is still acceptable.  More importantly, the calculation which does not use the
complicated S-curvature curvature formula for Randers metric or $(\alpha,\beta)$-metrics at all. Hopefully, With the help of (\ref{S-curvature-formula-2}), we can obtain the explicit formula of S-curvature for  many more generalized classes of homogeneous Finsler metrics.

Let $F=\alpha+\beta$ be a homogeneous Randers metric on $M=G/H$, where $H$ is the isotropy group of $x\in TM$. Let $\mathfrak{m}$ be a complement of $\mathfrak{h}$ in $\mathfrak{g}$. Identify $\mathfrak{m}$  with $TM_x$ as above.
Then $\alpha$ is determined by an linear metric on $\mathfrak{m}$ and $\beta$ is determined by a vector in $\mathfrak{m}^*$, which is invariant under the adjoint action of $H$. We still denote them as $\alpha$ and $\beta$. Let $\langle,\rangle$ be the inner product induced by $\alpha$ gives on $\mathfrak{m}$ and suppose $\beta$ is defined by $\beta(\cdot)=\langle\cdot,u\rangle$, where $u\in\mathfrak{m}$ is a fixed vector of $H$ (see \cite{De}).

For each $v\in TM_x=\mathfrak{m}$ with $\alpha(v)=1$,  we choose an orthonormal basis $\{\hat{v}_1,\ldots,\hat{v}_n\}$ of $\mathfrak{m}$ with respect to $\alpha$, such that $u=b\hat{v}_1$ and $v=a\hat{v}_1+a'\hat{v}_2$ with $a^2+a'^2=1$. Then $\alpha$ on $\mathfrak{m}$ is simplified as $\alpha=\sqrt{{y^1}^2+\cdots+{y^n}^2}$ and $\beta=by^1$.
All the notations are defined for this Killing
frame around $x$ defined by the $\hat{v}_i$s.

The fundamental tensor at $v$ is given by
\begin{eqnarray*}
g_{11} &=& 1+b^2+3ba-ba^3, \\
g_{12} &=& ba'^3, \\
g_{22} &=& 1+ b a^3, \\
g_{ii} &=& 1+ b a,
\end{eqnarray*}
and all other $g_{ij}$ vanishes at $v$.
Similarly, the tensor $g^{ij}$ is given by
\begin{eqnarray*}
g^{11} &=& (1+ba)^{-3} (1+ba^3),\\
g^{12} &=& -(1+ba)^{-3} ba'^3,\\
g^{22} &=& (1+ba)^{-3} (1+b^2+3ba-ba^3),\\
g^{ii} &=& (1+ba)^{-1},
\end{eqnarray*}
and all other $g^{ij}$ vanishes at $v$.
The coefficients of the Cartan torsion are then given by
\begin{eqnarray*}
I_1 &=& \frac{n+1}{2(1+ba)}ba'^2,\\
I_2 &=& -\frac{n+1}{2(1+ba)}baa',
\end{eqnarray*}
and all other $I_i$s vanish as $v$.
Therefore  at the vector $v$, we have
\begin{eqnarray*}
\nabla^{g_{ij}}\ln\sqrt{\det(g_{pq})}&=&\frac{(n+1)ba'^2}{2(1+ba)^3}\hat{v}_1-
\frac{(n+1)ba'(a+b)}{2(1+ba)^3}\hat{v}_2\nonumber\\
\label{4} &=&\frac{n+1}{2(1+ba)^2}u-\frac{(n+1)b(a+b)}{2(1+ba)^3}v,
\end{eqnarray*}
and
\begin{equation}\label{5}
\langle [v,\nabla^{g_{ij}}\ln\sqrt{\det(g_{pq})}(v)]_\mathfrak{m},v \rangle_v
=\frac{n+1}{2(1+ba)^2}\langle [v,u]_\mathfrak{m},v\rangle_v.
\end{equation}

The calculation of the fundamental tensor at $v$ indicates that
\begin{eqnarray*}
\langle [v,u]_\mathfrak{m},v \rangle_v =
ac g_{11}+(ad+a'c)g_{12}+a'd g_{22},
\end{eqnarray*}
where $[v,u]_{\mathfrak{m}}=c\hat{v}_1+d\hat{v}_2$, i.e.,
\begin{eqnarray}
\langle [y,u]_\mathfrak{m},u \rangle &=& bc, \\
\langle [y,u]_{\mathfrak{m}},v \rangle &=& ac+a'd.
\end{eqnarray}
Thus we have $a'bd=b\langle [v,u]_{\mathfrak{m}},v \rangle-a\langle [v,u]_\mathfrak{m},u \rangle$, and
\begin{eqnarray}
& &\langle [v,\nabla^{g_{ij}}\ln\sqrt{\det(g_{pq})}(v)]_\mathfrak{m},v \rangle_v\nonumber\\
&=& \frac{n+1}{2(1+ba)^2}[ac(1+b^2+3ba-ba^3)+(ad+a'c)ba'^3+a'd(1+ba^3)]\nonumber\\
&=& \frac{n+1}{2(1+ba)^2}[(ac+a'd)+(ab+3a^2-a^4-a^2 a'^2+a'^4)bc+(aa'^2+a^3)abd]\nonumber\\
&=& \frac{n+1}{2(1+ba)}(\langle [v,u]_\mathfrak{m},u \rangle+
\langle [v,u]_{\mathfrak{m}},v \rangle).
\end{eqnarray}
Notice that $F(v)=1+ba$ and $\alpha(v)=1$. Therefore for general $v$,
the homogeneity of the S-curvature indicates that the formula should be adjusted to
\begin{equation}
S(x,v)=\frac{n+1}{2F(v)}(\alpha(v)\langle [v,u]_\mathfrak{m},u \rangle+
\langle [v,u]_{\mathfrak{m}},v \rangle).
\end{equation}
This formula coincides with the one obtained  in \cite{De}.

\end{document}